\documentclass[12pt]{article}
\usepackage{amsthm}
\usepackage[margin=2cm]{geometry}
\usepackage[fleqn]{amsmath}
\usepackage{amssymb}
\usepackage{enumitem}
\usepackage{color}
\usepackage{tikz}
\usetikzlibrary{decorations.markings}

\theoremstyle{plain}
\newtheorem{theorem}{Theorem}[section]
\newtheorem{corollary}[theorem]{Corollary}
\newtheorem{lemma}[theorem]{Lemma}

\newtheorem{conjecture}[theorem]{Conjecture}
\theoremstyle{definition}

\DeclareMathOperator{\diam}{diam}

\DeclareMathOperator{\trace}{Tr}

\begin{document}
\title{On total regularity of mixed graphs with order close to the Moore bound}
\author{
James Tuite\thanks{Open University, Milton Keynes, UK}\\ \texttt{\small james.tuite@open.ac.uk}
\and Grahame Erskine \footnotemark[1]\\ \texttt{\small grahame.erskine@open.ac.uk}
}

\date{}

\maketitle
\let\thefootnote\relax\footnote{Mathematics subject classification: 05C35}
\let\thefootnote\relax\footnote{Keywords: degree-diameter problem, mixed graphs, almost Moore graphs, total regularity}

\vspace*{-8ex}
\begin{abstract}\noindent
The undirected degree/diameter and degree/girth problems and their directed analogues have been studied for many decades in the search for efficient network topologies.  Recently such questions have received much attention in the setting of mixed graphs, i.e. networks that admit both undirected \emph{edges} and directed \emph{arcs}. The degree/diameter problem for mixed graphs asks for the largest possible order of a network with diameter $k$, maximum undirected degree $\leq r$ and maximum directed out-degree $\leq z$.  It is also of interest to find smallest possible $k$-geodetic mixed graphs with minimum undirected degree $\geq r$ and minimum directed out-degree $\geq z$.  A simple counting argument reveals the existence of a natural bound, the \emph{Moore bound}, on the order of such graphs; a graph that meets this limit is a \emph{mixed Moore graph}.  Mixed Moore graphs can exist only for $k = 2$ and even in this case it is known that they are extremely rare.  It is therefore of interest to search for graphs with order one away from the Moore bound.  Such graphs must be out-regular; a much more difficult question is whether they must be totally regular.  For $k = 2$, we answer this question in the affirmative, thereby resolving an open problem stated in a recent paper of L\'opez and Miret.  We also present partial results for larger $k$.  We finally put these results to practical use by proving the uniqueness of a 2-geodetic mixed graph with order exceeding the Moore bound by one.
\end{abstract}

\tikzset{middlearrow/.style={
        decoration={markings,
            mark= at position 0.9 with {\arrow[scale=2]{#1}} ,
        },
        postaction={decorate}
    }
}

%----------------------------------------------
\section{Introduction}

The degree/diameter and degree/girth problems have their roots in the design of efficient interconnection networks.  Applications include the design of computer architectures, planning transportation infrastructure and analysis of social networks.  The undirected degree/diameter problem asks for the maximum possible order of a graph with given maximum degree $d$ and diameter $k$.  The order of such a graph is bounded above by the \emph{Moore bound}
\[ M(d,k) = 1 + d + d(d-1) + d(d-1)^2 + \dots + d(d-1)^{k-1}.\]

The degree/girth problem asks for the smallest graph with minimum degree $d$ and girth $\geq 2k+1$.  For such graphs the Moore bound serves as a lower bound.

A graph attaining this bound is known as a \emph{Moore graph}.  A simple counting argument shows that a graph $G$ is Moore if and only if it is $d$-regular, has diameter $k$ and girth $2k+1$, or, equivalently, if and only if $G$ is $d$-regular and between any pair of vertices $u,v$ there is a unique path of length $\leq k$.  
It was shown by spectral methods in~\cite{BanIto,HofSin} that for $k \geq 2$ and $d \geq 3$ such a graph must have diameter $k = 2$ and degree $d\in \{3,7,57\}$.  There are unique Moore graphs corresponding to $d = 3$ and $7$, given by the Petersen graph and the Hoffman-Singleton graph respectively.  The existence of a Moore graph corresponding to $d=57$ is a famous open problem.

Given the scarcity of Moore graphs, it is of great interest to study graphs with similar structures.  One can search for graphs with maximum degree $d$, diameter $k$ and order $M(d,k)-\delta $ for small $\delta $; $\delta $ is known as the \emph{defect} of the graph.  Alternatively, one can ask for graphs with girth $g \geq 2k + 1$, minimum degree $d$ and order $M(d,k)+\epsilon $, where $\epsilon $ is the (hopefully small) \emph{excess} of the graph.  Unfortunately graphs with defect or excess one exist only trivially in the form of cycles~\cite{BanIto2,ErdFajHof,KurTsu}.  For more information on the history and development of the degree/diameter and degree/girth problems, see the surveys~\cite{ExoJaj,MilSir}.

For digraphs with maximum out-degree $d>1$ and diameter $k>1$ the Moore bound has an even simpler form:
\[M(d,k)= 1+d+d^2+\dots +d^k = \frac{d^{k+1}-1}{d-1}.\]
A digraph is Moore if and only if it is out-regular with degree $d$, has diameter $k$ and is $k$-geodetic, i.e. for any pair of vertices $u,v$ there is at most one directed walk from $u$ to $v$ with length $\leq k$.  It is well known~\cite{BriTou,PlesZnam} that there are no Moore digraphs of degree $d \geq 2$ and diameter $k \geq 2$.  There are digraphs with diameter $k = 2$ and defect $\delta = 1$ for every value of the maximum out-degree $d$~\cite{FioYebAle}; however, no such digraphs exist for diameters $k = 3$ or $4$~\cite{ConGimGonMirMor,ConGimGonMirMor2} or degrees $d = 2$ or $3$ and $k \geq 3$~\cite{BasMilSirSut,MilFri}.  $k$-geodetic digraphs with minimum out-degree $d$ and excess $\epsilon = 1$ do not exist for degree $d = 2$~\cite{Sil} or $k = 3$ or $4$ and $d \geq 2$ or $k = 2$ and $d \geq 8$~\cite{MilMirSil}.  In the directed problem, it is necessary to ask whether digraphs with order one away from the Moore bound are necessarily diregular; this was settled in the affirmative for defect one in~\cite{MilGimSirSla} and for excess one in~\cite{Sil}.

Recently, there has been much interest in these problems in the context of \emph{mixed} graphs, where we allow both undirected edges and directed arcs in the graph.  We can view the case of mixed graphs either as a generalisation of the undirected case (allowing arcs as well as edges) or as a specialisation of the directed case (where we insist that a number of the arcs must be present with their reverses).  The mixed Moore bound for undirected degree $r$ and directed out-degree $z$ is given by~\cite{BusAmiErsMilPer}
\[ M(r,z,k) = A\frac{u_1^{k+1}-1}{u_1-1}+B\frac{u_2^{k+1}-1}{u_2-1},\]
where  $v = (z+r)^2+2(z-r)+1$ and
\[ u_1 = \frac{z+r-1-\sqrt{v}}{2}, u_2 = \frac{z+r-1+\sqrt{v}}{2}\]
\[ A = \frac{\sqrt{v}-(z+r+1)}{2\sqrt{v}}, B = \frac{\sqrt{v}+(z+r+1)}{2\sqrt{v}}.\]
This serves as a natural upper bound for mixed graphs with maximum undirected degree $r$, maximum directed out-degree $z$ and diameter $k$, and as a lower bound for $k$-geodetic mixed graphs with minimum undirected degree $r$ and minimum directed out-degree $z$.  Note that we do not explicitly constrain the in-degree of any vertex.  It is known~\cite{NguMilGim} that no mixed Moore graph of diameter $k \geq 3$ can exist.  For diameter $k = 2$ the Moore bound is  
\[M(r,z,2)=(r+z)^2+z+1.\]
There is a mixed Moore graph with diameter $k = 2$ and undirected degree $r = 1$ for every value of the out-degree $z$; this graph is obtained by collapsing digons in the Kautz digraphs into edges~\cite{NguMilGim}.  Otherwise, just three sporadic mixed Moore graphs have been identified~\cite{Bos,Jor}.  For diameter $k = 2$, a strong necessary condition on the parameters $r$ and $z$ follows from analysis of the eigenvalues of the graph's adjacency matrix~\cite{Bos}.  However, there are infinitely many pairs $r,z$ for which the existence of a Moore graph is not known.  

A mixed graph with undirected degree $r$, directed out-degree $z$, diameter $k$ and defect $\delta = 1$ is known as an \emph{almost mixed Moore graph}.  Only one almost mixed Moore graph is known to exist~\cite{LopMir}; this graph has order $n = 10$ and parameters $r=2, z=1$ and $k = 2$ and is shown in Figure~\ref{fig:almost}.  This graph has the property of total regularity, i.e. its undirected subgraph is regular and its directed subgraph is diregular.  In~\cite{BusLopMir} it is shown that this is the unique almost mixed Moore graph with these parameters; the proof is complicated by the fact that any such graph could not be assumed to be totally regular.  These considerations prompted L\'opez and Miret to ask whether almost mixed Moore graphs with diameter two must be totally regular~\cite{LopMir}.  In this paper we answer this question in the affirmative, thereby greatly simplifying the search for almost mixed Moore graphs.  We also consider the case $r = z =1$ and larger diameters.

The existence of $k$-geodetic mixed graphs with undirected degree $\geq d$, directed out-degree $\geq z$ and order $M(r,z,k)+1$ is also an intriguing problem.  Conditions on $r$ and $z$ for the existence of such graphs for $k = 2$ are given in~\cite{Tui} and small $k$-geodetic mixed graphs are presented in \cite{TuiErs}, along with lower bounds on the order of totally regular $k$-geodetic mixed graphs.  We show that $k$-geodetic mixed graphs with undirected degree $\geq r$, directed out-degree $\geq z$ and excess one are totally regular for $k = 2$; and for $k \geq 3$ we prove total regularity for mixed graphs with directed out-degree $z = 1$.  We also present a $2$-geodetic mixed graph with $r = 2, z = 1$ and excess $\epsilon = 1$ and use our total regularity result to prove its uniqueness.

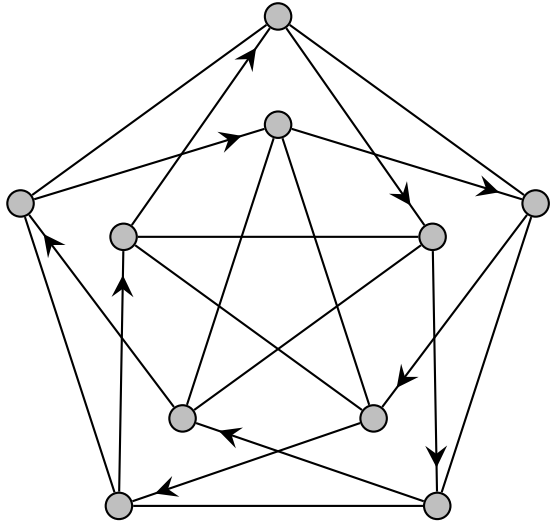
\begin{figure}
\centering
\begin{tikzpicture}[middlearrow=stealth,x=0.2mm,y=-0.2mm,inner sep=0.2mm,scale=1,thick,vertex/.style={circle,draw,minimum size=10,fill=lightgray}]
\node at (380,200) [vertex] (v1) {};
\node at (208.8,324.4) [vertex] (v2) {};
\node at (274.2,525.6) [vertex] (v3) {};
\node at (485.8,525.6) [vertex] (v4) {};
\node at (551.2,324.4) [vertex] (v5) {};
\node at (380,272) [vertex] (v6) {};
\node at (316.5,467.4) [vertex] (v7) {};
\node at (482.7,346.6) [vertex] (v8) {};
\node at (277.3,346.6) [vertex] (v9) {};
\node at (443.5,467.4) [vertex] (v10) {};
\path
	(v1) edge (v2)
	(v1) edge (v5)
	(v2) edge (v3)
	(v3) edge (v4)
	(v4) edge (v5)
	(v6) edge (v7)
	(v6) edge (v10)
	(v7) edge (v8)
	(v8) edge (v9)
	(v9) edge (v10)
	(v2) edge [middlearrow] (v6)
	(v6) edge [middlearrow] (v5)
	(v5) edge [middlearrow] (v10)
	(v10) edge [middlearrow] (v3)
	(v3) edge [middlearrow] (v9)
	(v9) edge [middlearrow] (v1)
	(v1) edge [middlearrow] (v8)
	(v8) edge [middlearrow] (v4)
	(v4) edge [middlearrow] (v7)
	(v7) edge [middlearrow] (v2)
	;
\end{tikzpicture}
\caption{A mixed almost Moore graph}
\label{fig:almost}
\end{figure}

%----------------------------------------------
\section{Total regularity of almost mixed Moore graphs with diameter two}

We begin by establishing our notation.  A mixed graph $G$ consists of a set $V(G)$ of vertices, a set $E(G)$ of unordered pairs of vertices called \emph{edges} and a set $A(G)$ of ordered pairs called \emph{arcs}.  If $(u,v) \in E$, then we write $u \sim v$; similarly $(u,v) \in A$ is denoted by $u \to v$.  We do not allow loops or parallel edges or arcs; hence for no vertex $u$ can we have $u \sim u$ or $u \to u$, and $u \sim v$ and $u \to v$ cannot hold simultaneously.  The undirected subgraph $G^U$ and directed subgraph $G^Z$ of a mixed graph $G$ are the subgraphs induced by $E(G)$ and $A(G)$ respectively.  The distance $d(u,v)$ between vertices $u$ and $v$ of $G$ is the length of a shortest path in $G$ from $u$ to $v$; observe that we can have $d(u,v) \not = d(v,u)$.  The diameter of $G$ is defined to be $\diam(G)=\max\{ d(u,v) : u,v \in V(G)\} $.  The distance from $u$ to $v$ in the undirected and directed subgraphs will be written as $d_U(u,v)$ and $d_Z(u,v)$ respectively; we will leave these quantities undefined if $u$ and $v$ lie in separate components of $G^U$ or $G^Z$.  $G$ will always be assumed to be a proper mixed graph, i.e. both $E(G)$ and $A(G)$ are non-empty.  A \emph{walk} $u_0u_1u_2\dots u_t$ of length $t$ in a mixed graph $G$ is an alternating sequence of vertices of $G$ such that for $0 \leq i \leq t-1$ we have $u_i \sim u_{i+1}$ or $u_i  \to u_{i+1}$.  The walk is \emph{non-backtracking} if it does not contain a sub-walk $u_i\sim u_{i+1} \sim u_i$.  We will colloquially call a mixed non-backtracking walk a \emph{mixed path}. $G$ is \emph{$k$-geodetic} if for any ordered pair of (not necessarily distinct) vertices $u, v$ there is at most one mixed path from $u$ to $v$ with length $\leq k$.    

The undirected degree of a vertex $u$ will be denoted by $d(u)$, the directed out-degree of $u$ by $d^+(u)$ and the directed in-degree of $u$ by $d^-(u)$.  A mixed graph is \emph{out-regular} if $d(u) = d(v)$ and $d^+(u) = d^+(v)$ for all vertices $u,v$.  A graph is \emph{totally regular} if it is out-regular and $d^-(u) = d^+(u)$ for every vertex $u$.  Hence a mixed graph $G$ is totally regular if and only if $G^U$ is regular and $G^Z$ is diregular.  

For given maximum directed out-degree $z$, the sets $S$ and $S'$ are defined by
\[ S = \{ v \in V(G) : d^-(v) < z\};\quad S' = \{ v' \in V(G) : d^-(v') > z\} .\]
Informally, $S$ and $S'$ are the sets of vertices with `too small' in-degree and `too large' in-degree respectively.  The undirected and directed neighbourhoods of a vertex $u$ are defined by
\[ U(u) = \{ v \in V(G) : u \sim v\} , Z^+(u) = \{ v \in V(G) : u \to v\}, Z^-(u) = \{ v \in V(G) : v \to u\} .\]
For convenience, we also set
\[ N^+(u) = U(u) \cup Z^+(u) , N^-(u) = U(u) \cup Z^-(u).\]
Thus $N^+(u)$ is the set of vertices that can be reached by a path of length one from $u$ and $N^-(u)$ is the set of vertices that can reach $u$ by a path of length one.  

A graph with maximum undirected degree $r$, maximum directed out-degree $z$, diameter $k$ and defect $\delta $ (and hence order $n = M(r,z,k) - \delta$) is a $(r,z,k;-\delta)$-graph.  An $(r,z,k;0)$-graph is thus a mixed Moore graph and a $(r,z,k;-1)$-graph is an almost mixed Moore graph.  We recall some basic properties of mixed Moore and almost mixed Moore graphs.  To derive the Moore bound for a graph with undirected degree $\leq r$, directed out-degree $\leq z$ and diameter $k = 2$, we pick an arbitrary vertex $u$ and draw a \emph{Moore tree} (as illustrated in Figure~\ref{fig:mooretree} in the case $r=3,z=3$).  In the case of a mixed Moore graph, all paths of lengths $\leq 2$ from the root vertex $u$ reach distinct vertices, and a simple count yields the Moore bound of $M(r,z,2) = (r+z)^2+z+1$.

\begin{figure}
	\centering
	\begin{tikzpicture}[middlearrow=stealth,x=0.2mm,y=-0.2mm,inner sep=0.1mm,scale=2.1,
	thick,vertex/.style={circle,draw,minimum size=10,font=\small,fill=lightgray},every label/.style={font=\scriptsize}]
	
	\node at (200,0) [vertex,label=above:{$0$}] (v0) {};
	
	\node at (25,100) [vertex,label=left:{$1$}] (v1) {};
	\node at (95,100) [vertex,label=left:{$2$}] (v2) {};
	\node at (165,100) [vertex,label=left:{$3$}] (v3) {};
	\node at (235,100) [vertex,label=left:{$4$}] (v4) {};
	\node at (305,100) [vertex,label=left:{$5$}] (v5) {};
	\node at (375,100) [vertex,label=left:{$6$}] (v6) {};
	
	\node at (0,200) [vertex,label=below:{$7$}] (v11) {};
	\node at (10,200) [vertex,label=below:{$8$}] (v12) {};
	\node at (20,200) [vertex,label=below:{$9$}] (v13) {};
	\node at (30,200) [vertex,label=below:{$10$}] (v14) {};
	\node at (40,200) [vertex,label=below:{$11$}] (v15) {};
	\node at (50,200) [vertex,label=below:{$12$}] (v16) {};
	\node at (70,200) [vertex,label=below:{$13$}] (v21) {};
	\node at (80,200) [vertex,label=below:{$14$}] (v22) {};
	\node at (90,200) [vertex,label=below:{$15$}] (v23) {};
	\node at (100,200) [vertex,label=below:{$16$}] (v24) {};
	\node at (110,200) [vertex,label=below:{$17$}] (v25) {};
	\node at (120,200) [vertex,label=below:{$18$}] (v26) {};
	\node at (140,200) [vertex,label=below:{$19$}] (v31) {};
	\node at (150,200) [vertex,label=below:{$20$}] (v32) {};
	\node at (160,200) [vertex,label=below:{$21$}] (v33) {};
	\node at (170,200) [vertex,label=below:{$22$}] (v34) {};
	\node at (180,200) [vertex,label=below:{$23$}] (v35) {};
	\node at (190,200) [vertex,label=below:{$24$}] (v36) {};
	\node at (210,200) [vertex,label=below:{$25$}] (v41) {};
	\node at (220,200) [vertex,label=below:{$26$}] (v42) {};
	\node at (230,200) [vertex,label=below:{$27$}] (v43) {};
	\node at (240,200) [vertex,label=below:{$28$}] (v44) {};
	\node at (250,200) [vertex,label=below:{$29$}] (v45) {};
	\node at (280,200) [vertex,label=below:{$30$}] (v51) {};
	\node at (290,200) [vertex,label=below:{$31$}] (v52) {};
	\node at (300,200) [vertex,label=below:{$32$}] (v53) {};
	\node at (310,200) [vertex,label=below:{$33$}] (v54) {};
	\node at (320,200) [vertex,label=below:{$34$}] (v55) {};
	\node at (350,200) [vertex,label=below:{$35$}] (v61) {};
	\node at (360,200) [vertex,label=below:{$36$}] (v62) {};
	\node at (370,200) [vertex,label=below:{$37$}] (v63) {};
	\node at (380,200) [vertex,label=below:{$38$}] (v64) {};
	\node at (390,200) [vertex,label=below:{$39$}] (v65) {};
	
	\path
	(v0) edge [middlearrow] (v1)
	(v0) edge [middlearrow] (v2)
	(v0) edge [middlearrow] (v3)
	(v0) edge (v4)
	(v0) edge (v5)
	(v0) edge (v6)
	(v1) edge [middlearrow] (v11)
	(v1) edge [middlearrow] (v12)
	(v1) edge [middlearrow] (v13)
	(v1) edge (v14)
	(v1) edge (v15)
	(v1) edge (v16)
	(v2) edge [middlearrow] (v21)
	(v2) edge [middlearrow] (v22)
	(v2) edge [middlearrow] (v23)
	(v2) edge (v24)
	(v2) edge (v25)
	(v2) edge (v26)
	(v3) edge [middlearrow] (v31)
	(v3) edge [middlearrow] (v32)
	(v3) edge [middlearrow] (v33)
	(v3) edge (v34)
	(v3) edge (v35)
	(v3) edge (v36)
	(v4) edge [middlearrow] (v41)
	(v4) edge [middlearrow] (v42)
	(v4) edge [middlearrow] (v43)
	(v4) edge (v44)
	(v4) edge (v45)
	(v5) edge [middlearrow] (v51)
	(v5) edge [middlearrow] (v52)
	(v5) edge [middlearrow] (v53)
	(v5) edge (v54)
	(v5) edge (v55)
	(v6) edge [middlearrow] (v61)
	(v6) edge [middlearrow] (v62)
	(v6) edge [middlearrow] (v63)
	(v6) edge (v64)
	(v6) edge (v65)
	;
	\end{tikzpicture}
	\caption{The Moore tree for $z=3,r=3,k=2$}
	\label{fig:mooretree}
\end{figure}
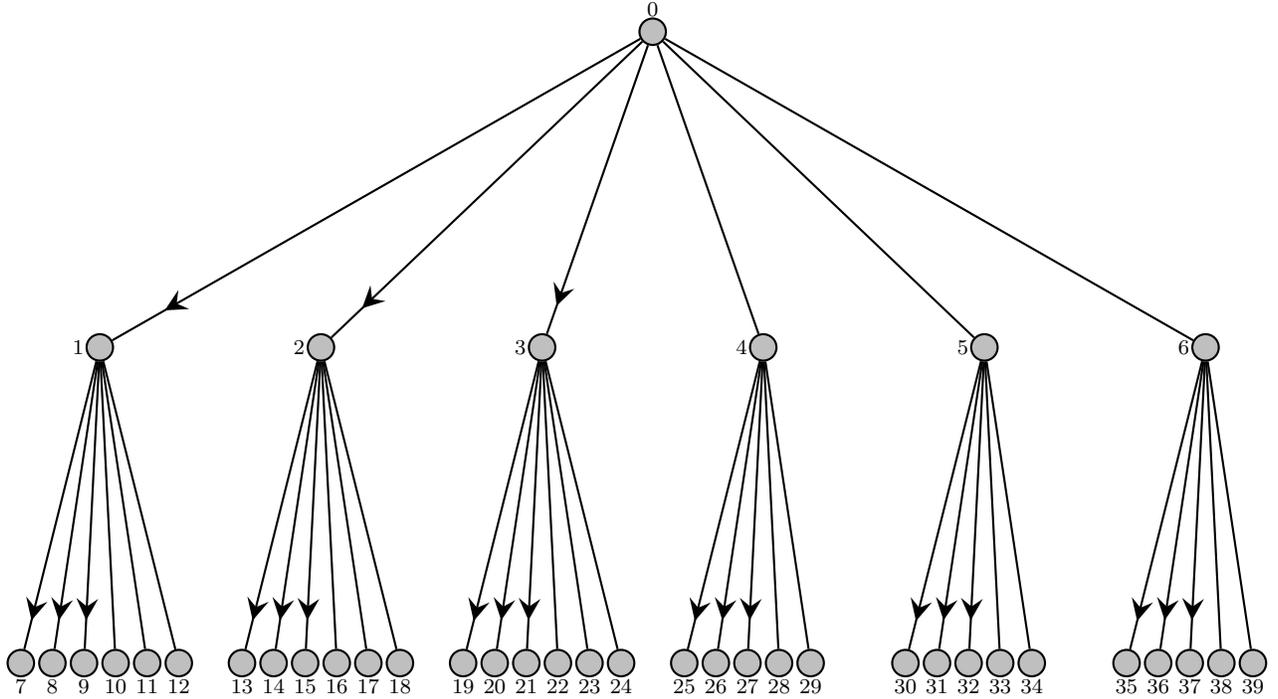

In the case of an almost mixed Moore graph $G$, i.e. a graph with defect $\delta = 1$ and hence order $n = (r+z)^2+z$, for each vertex $u$ there must be exactly one vertex $v$ that appears twice in the Moore tree rooted at $u$; we write $v = r(u)$ and call $v$ the \emph{repeat} of $u$.  The repeat function is an automorphism of $G$ if and only if $G$ is totally regular~\cite{LopMir}. If $W$ is a subset of $V(G)$ we denote the set of repeats of vertices in $W$ by $r(W)$.

We will now proceed to show that any $(r,z,2;-1)$-graph is totally regular.  Suppose that $G$ is an $(r,z,2;-1)$-graph that is not totally regular.  Our strategy is to use a purely combinatorial argument to deduce structural information about $G$ and then apply spectral theory to obtain a contradiction.  First we show that $G$ must be out-regular.

\begin{lemma}\label{lem:outregular}
$G$ is out-regular with undirected degree $r$ and directed out-degree $z$.
\end{lemma} 
\begin{proof}
Suppose that for some vertex $u$, $d(u) \leq r-1$ or $d^+(u) \leq z-1$.  Drawing the Moore tree rooted at $u$,  it is evident that the defect of $G$ is at least $r+z \geq 2$.  
\end{proof}  

We now prove two fundamental lemmas that show the relationship between the sets $S$, $S'$, out-neighbourhoods and repeats.  The result that the repeat function is an automorphism for totally regular almost mixed Moore graphs can be viewed as a `limiting case' of these results.

\begin{lemma}\label{lem:fundamental}
If $v \in S$, then $d^-(v) = z-1$ and for all $u \in V(G)$ we have $S \subseteq N^+(r(u))$.
\end{lemma}
\begin{proof}
Let $v \in S$ and $u \in V(G)$.  Consider the Moore tree rooted at $u$.  Suppose that $d(u,v) = 2$.  As each vertex of $N^+(u)$ can reach $v$ by paths of length $\leq 2$, each branch of the Moore tree must contain an element of $N^-(v)$.  However, there are $r + z$ branches, whereas $v$ has $\leq r+z-1$ in-neighbours, so at least one in-neighbour must be repeated in the tree.  Hence $v$ is an out-neighbour of $r(v)$.  Evidently $d^-(v) = z-1$, for otherwise more than one in-neighbour of $v$ would be repeated and the defect of $G$ would be at least two.  The cases $u \sim v, u \to v$ and $u = v$ can be dealt with similarly. 
\end{proof}

\begin{corollary}\label{cor:averageindegree}
$\displaystyle\sum _{v' \in S' }(d^-(v')-z) = \sum _{v \in S} (z - d^-(v)) = |S|$.
\end{corollary}
\begin{proof}
By Lemma~\ref{lem:outregular}, the average directed in-degree must be $z$.  The final equality follows from Lemma~\ref{lem:fundamental}.  
\end{proof}

\begin{lemma}\label{lem:fundamental2}
For all $u \in V(G)$ we have $S' \subseteq r(N^+(u))$.
\end{lemma}
\begin{proof}
Let $v' \in S', u \in V(G)$.  Suppose that $d(u,v') = 2$.  Then $u \not \in N^-(v')$.  There are $r+z$ branches of the Moore tree at $u$, but $v'$ has $\geq r+z+1$ in-neighbours, so at least one branch contains more than one in-neighbour of $v'$, so that $v' \in r(N^+(u))$.  The remaining cases are similar. 
\end{proof}

Lemmas~\ref{lem:fundamental} and~\ref{lem:fundamental2} not only yield important information on the structure of $G$, but also show that the order of both sets $S,S'$ is bounded above by $r+z$.  A counting argument will now allow us to ascertain the exact size of $S$.

\begin{lemma}
$|S| = r+z$.
\end{lemma}
\begin{proof}
Let $v \in S$.  By Lemma~\ref{lem:fundamental} we have $d^-(v) = z-1$.  We obtain an upper bound on the number of vertices that are initial points of paths of length $\leq 2$ that terminate at $v$ by assuming that $S' \subseteq N^-(v)$.  As $G$ has diameter two, this yields by Corollary \ref{cor:averageindegree}   
\[ n \leq 1+r+(z-1)+r(r+z-1)+(z-1)(r+z)+|S|. \]
Rearranging, $|S| \geq r+z$.  Combined with the result of Lemma~\ref{lem:fundamental}, we see that $|S| = r+z$. 
\end{proof}

\begin{corollary}\label{cor:neighbourhood}
$S = N^+(r(u))$ for all $u \in V(G)$.
\end{corollary}

We say that a vertex $w$ is a \emph{repeat} in $G$ if there exists a vertex $u$ such that $r(u) = w$.  The preceding corollary allows us to determine both the value of the undirected degree $r$ and the number of distinct repeats in $G$.

\begin{lemma}\label{lem:2repeats}
$r = 2$ and there are exactly two repeats.
\end{lemma}
\begin{proof}
Suppose that there is only one repeat, call it $R$.  As $R$ is the only repeat, we must have $r(R) = R$.  It is easily seen that there must be $u \in N^+(R)$ such that $R \to u \to R$.  But $u$ now lies in a 2-cycle and hence is a repeat, contradicting our hypothesis.  Hence $G$ has at least two repeats, call them $R_1$ and $R_2$.

If $r = 1$, then $G$ contains a perfect matching, so that $G$ must have even order, whereas $|V(G)| = z^2+3z+1$ would be odd.  Suppose that $r \geq 3$.  By Corollary~\ref{cor:neighbourhood}, $N^+(R_1) = N^+(R_2) = S$.  $R_1$ thus has at least $r \geq 3$ paths of length two to $R_2$, as shown in Figure \ref{fig:r3}, contradicting $\delta = 1$. Therefore $r = 2$.  By the foregoing reasoning, $r(R_1) = R_2$ and $r(R_2) = R_1$.  If there were a third repeat $R_3$, this argument could be repeated with $R_3$ in place of $R_2$ to give $r(R_1) = R_3$, so that $R_2 = R_3$, a contradiction.  Therefore there are exactly two repeats $R_1$ and $R_2$.    
\end{proof}

\begin{figure}\centering
	\begin{tikzpicture}[middlearrow=stealth,x=0.2mm,y=-0.2mm,inner sep=0.1mm,scale=1.35,
	thick,vertex/.style={circle,draw,minimum size=10,font=\small,fill=lightgray},every label/.style={font=\small}]
	
	\node at (200,0) [vertex,label=above:{$R_1$}] (v0) {};
	\node at (25,100) [vertex,label=left:{$v_1$}] (v1) {};
	\node at (95,100) [vertex,label=left:{$v_2$}] (v2) {};
	\node at (165,100) [vertex,label=left:{$v_3$}] (v3) {};
	\node at (235,100) [vertex,label=left:{$v_4$}] (v4) {};
	\node at (305,100) [vertex,label=left:{$v_5$}] (v5) {};
	\node at (375,100) [vertex,label=left:{$v_6$}] (v6) {};
	\node at (200,200) [vertex,label=below:{$R_2$}] (v7) {};

	\path
	(v0) edge [middlearrow] (v1)
	(v0) edge [middlearrow] (v2)
	(v0) edge [middlearrow] (v3)
	(v0) edge (v4)
	(v0) edge (v5)
	(v0) edge (v6)
	(v7) edge (v1)
	(v7) edge (v2)
	(v7) edge [middlearrow] (v3)
	(v7) edge [middlearrow] (v4)
	(v7) edge [middlearrow] (v5)
	(v7) edge (v6)
	;
	\end{tikzpicture}
	\caption{If $r \geq 3$}
	\label{fig:r3}
\end{figure}
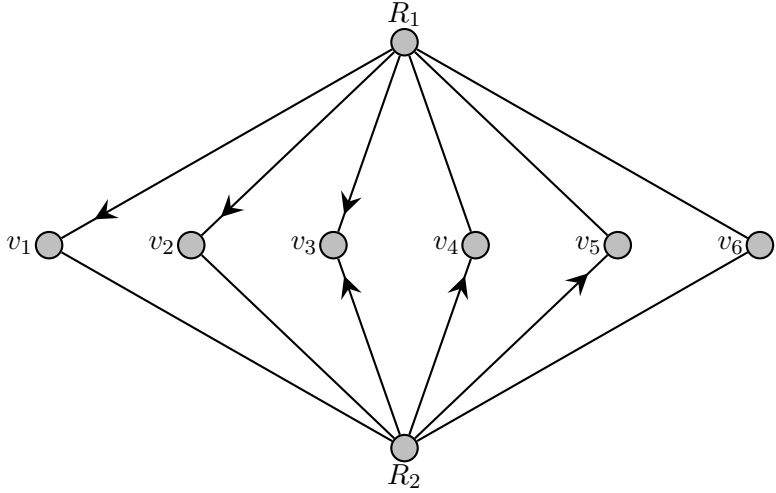

This provides us with all of the structural information necessary to complete our proof.  We now adopt a spectral approach.

\begin{theorem}\label{main theorem}
Almost mixed Moore graphs with diameter two are totally regular.
\end{theorem}
\begin{proof}
Suppose that there are $m_1$ vertices with repeat $R_1$ and $m_2$ vertices with repeat $R_2$; we shall call vertices with repeat $R_i$ \emph{Type i}.  Let us label the vertices of $G$ $u_1,u_2,\dots , u_n$, so that $u_1 = R_1, u_2 = R_2$, $u_{2+s}$ is Type 2 for $1 \leq s \leq m_2-1$ and $u_{1+m_2+t}$ is Type 1 for $1 \leq t \leq m_1-1$.  If $A$ is the adjacency matrix of $G$, the $(i,j)$-entry of the sum $I + A + A^2$ is the number of distinct walks of length $\leq 2$ from vertex $u_i$ to vertex $u_j$.  Hence
\[ I + A + A^2 = J + 2I + P,\]
where $I$ is the $n \times n$ identity matrix, $J$ is the $n \times n$ all-one matrix and $P$ is the matrix with entries given by 
\begin{equation}
  P_{ij}=\begin{cases}
    1, & \text{if $r(u_i) = u_j$}.\\
    0, & \text{otherwise}.
  \end{cases}
\end{equation}
Observe that all non-zero entries of $P$ occur in the first two columns.  Inspection shows that the matrix $J+P$ has the following eigenvalues:  
\begin{enumerate}[label=\roman*)]
	\item eigenvalue $n+1$ with multiplicity one.\newline
	Corresponding eigenvector: all-one vector of length $n$.
	\item eigenvalue $-1$ with multiplicity one.\newline 
	Corresponding eigenvector: 
	\begin{equation}
		f(u_j)=
			\begin{cases}
				1, & \text{if $u_j$ is Type 1 },\\
				-(m_1+1)/(m_2+1), & \text{if $u_j$ is Type 2 }.
			\end{cases}
	\end{equation} 
	
	\item eigenvalue $0$ with multiplicity $n-2$.\newline 
	Corresponding eigenvectors:
	\begin{equation}
		f_i(u_j)=
			\begin{cases}
				1, & \text{if $j = i+2$},\\
				-1, & \text{if $j=n$}, \\
				0, & \text{otherwise}. 
			\end{cases}
	\end{equation} 
	for $1 \leq i \leq n-3$ and
	\begin{equation}
		g(u_j)=
			\begin{cases}
				1, & \text{if $j = 1$ or $2$},\\
				-3, & \text{if $j = n$}, \\
				0, & otherwise.
			\end{cases}
	\end{equation} 
	
\end{enumerate}

It follows that $I+A+A^2$ has spectrum $\{ n+3, 1, 2^{(n-2)}\} $, so that $G$ must have spectrum $\sigma (G) = \{ \lambda_1 , \dots , \lambda_n\} $, where 
\begin{itemize}
\item $\lambda _1$ is a solution of $\lambda ^2 + \lambda -(n+2) = 0$
\item $\lambda _2$ is a solution of $\lambda ^2 + \lambda = 0$
\item $\lambda _i$ is a solution of $\lambda ^2 + \lambda - 1 = 0$ for $3 \leq i \leq n$.
\end{itemize}
The solutions of the first equation are $\lambda _1 = \frac{-1 \pm \sqrt{4n+9}}{2}$.  By Lemma~\ref{lem:2repeats}, the order of $G$ is $n = z^2+5z+4$, so
$4n+9 = 4z^2+20z+25 = (2z+5)^2$, yielding $\lambda _1 = z+2$ or $-z-3$.  Trivially $\lambda _2 = 0$ or $-1$.  Finally the third equation has solutions $\frac{-1 \pm \sqrt{5}}{2}$.  

$G$ has no loops, so the trace of its adjacency matrix is $\trace(A) = 0$.  It follows that the sum of the eigenvalues of $G$ is also zero.  In order that this sum be rational, one half of the eigenvalues $\lambda _3, \dots , \lambda _n$ must take the plus sign and half the negative sign.  The sum of the eigenvalues $\lambda_3, \dots , \lambda _n$ is thus $-(n/2)+ 1$.  Depending upon the values of $\lambda _1, \lambda _2$, this leaves us with four possibilities for the sum of the eigenvalues, namely $z+3-(n/2)$, $z+2-(n/2)$, $-z-2-(n/2)$ and $-z-3-(n/2)$.  The final two are clearly strictly negative, whilst the first two yield no feasible solutions for $z>0$.
\end{proof}

\section{Total regularity of almost mixed Moore graphs with $r = z = 1$ and diameter $k \geq 3$}

The problem of the total regularity of almost mixed Moore graphs with diameter $k \geq 3$ is difficult and we consider it here only in the case $r = z = 1$.  Let $G$ be a $(1,1,k;-1)$-graph that is not totally regular.  Trivially, $G$ is out-regular, so that every vertex $u$ has a unique undirected neighbour, which we will denote by $u^*$, and a unique directed out-neighbour, which will be written $u^+$.  For these parameters $S = \{ v \in V(G) : d^-(v) = 0\} $, $S' = \{ v' \in V(G) : d^-(v') \geq 2\} $; hence vertices in $S$ have no directed in-neighbours, i.e. $Z^-(v) = \varnothing $ for $v \in S$.  

Let $F_0 = F_1 = 1, F_2 = 2, F_3 = 3, F_4 = 5 \dots $ be the sequence of Fibonacci numbers.  It is easily verified that $G$ has order $n =  F_{k+3}-3$.  Since the order of $G$ must be even, no almost mixed Moore graph exists with these parameters for $k \equiv 2 \pmod 3$.  We will draw a Moore tree rooted at a vertex $u$ as shown in Figure \ref{fig:rz1Mooretree} for $k = 5$.  In general, for $0 \leq \ell \leq k-1$ the children in level $\ell+1$ of a vertex at level $\ell$ of the tree are drawn below it, with the undirected neighbour (if the vertex has its undirected neighbour at level $\ell+1$ and not $\ell-1$) to the left and the directed out-neighbour on the right, where vertices are labelled $u_0, u_1 \dots $ etc. in increasing  order from top to bottom and left to right.  This process is continued until the tree reaches depth $k$.

\begin{figure}\centering
	\begin{tikzpicture}[middlearrow=stealth,x=0.2mm,y=-0.2mm,inner sep=0.1mm,scale=1.35,
	thick,vertex/.style={circle,draw,minimum size=10,font=\small,fill=lightgray},every label/.style={font=\small}]
	
	\node at (0,0) [vertex,label=above:{$u_0$}] (u0) {};
	
	\node at (-150,60) [vertex,label=left:{$u_1$}] (u1) {};
	\node at (150,60) [vertex,label=right:{$u_2$}] (u2) {};
	
	\node at (-150,120) [vertex,label=left:{$u_3$}] (u3) {};
	\node at (50,120) [vertex,label=left:{$u_4$}] (u4) {};
	\node at (250,120) [vertex,label=right:{$u_5$}] (u5) {};

	\node at (-200,180) [vertex,label=left:{$u_6$}] (u6) {};
	\node at (-100,180) [vertex,label=right:{$u_7$}] (u7) {};
	\node at (50,180) [vertex,label=left:{$u_8$}] (u8) {};
	\node at (200,180) [vertex,label=left:{$u_9$}] (u9) {};
	\node at (300,180) [vertex,label=right:{$u_{10}$}] (u10) {};

	\node at (-200,240) [vertex,label=left:{$u_{11}$}] (u11) {};
	\node at (-125,240) [vertex,label=left:{$u_{12}$}] (u12) {};
	\node at (-75,240) [vertex,label=right:{$u_{13}$}] (u13) {};
	\node at (25,240) [vertex,label=left:{$u_{14}$}] (u14) {};
	\node at (75,240) [vertex,label=right:{$u_{15}$}] (u15) {};
	\node at (200,240) [vertex,label=left:{$u_{16}$}] (u16) {};
	\node at (275,240) [vertex,label=left:{$u_{17}$}] (u17) {};
	\node at (325,240) [vertex,label=right:{$u_{18}$}] (u18) {};

	\node at (-215,300) [vertex,label=below:{$u_{19}$}] (u19) {};
	\node at (-185,300) [vertex,label=below:{$u_{20}$}] (u20) {};
	\node at (-125,300) [vertex,label=below:{$u_{21}$}] (u21) {};
	\node at (-90,300) [vertex,label=below:{$u_{22}$}] (u22) {};
	\node at (-60,300) [vertex,label=below:{$u_{23}$}] (u23) {};
	\node at (25,300) [vertex,label=below:{$u_{24}$}] (u24) {};
	\node at (60,300) [vertex,label=below:{$u_{25}$}] (u25) {};
	\node at (90,300) [vertex,label=below:{$u_{26}$}] (u26) {};
	\node at (185,300) [vertex,label=below:{$u_{27}$}] (u27) {};
	\node at (215,300) [vertex,label=below:{$u_{28}$}] (u28) {};
	\node at (275,300) [vertex,label=below:{$u_{29}$}] (u29) {};
	\node at (310,300) [vertex,label=below:{$u_{30}$}] (u30) {};
	\node at (340,300) [vertex,label=below:{$u_{31}$}] (u31) {};

	\path
	(u0) edge (u1)
	(u0) edge [middlearrow] (u2)
	
	(u1) edge [middlearrow] (u3)
	(u2) edge (u4)
	(u2) edge [middlearrow] (u5)
	
	(u3) edge (u6)
	(u3) edge [middlearrow] (u7)
	(u4) edge [middlearrow] (u8)
	(u5) edge (u9)
	(u5) edge [middlearrow] (u10)
	
	(u7) edge (u12)
	(u8) edge (u14)
	(u10) edge (u17)
	(u6) edge [middlearrow] (u11)
	(u7) edge [middlearrow] (u13)
	(u8) edge [middlearrow] (u15)
	(u9) edge [middlearrow] (u16)
	(u10) edge [middlearrow] (u18)

	(u11) edge (u19)
	(u11) edge [middlearrow] (u20)
	(u12) edge [middlearrow] (u21)
	(u13) edge (u22)
	(u13) edge [middlearrow] (u23)
	(u14) edge [middlearrow] (u24)
	(u15) edge (u25)
	(u15) edge [middlearrow] (u26)
	(u16) edge (u27)
	(u16) edge [middlearrow] (u28)
	(u17) edge [middlearrow] (u29)
	(u18) edge (u30)
	(u18) edge [middlearrow] (u31)
	;
	\end{tikzpicture}
	\caption{The Moore tree for $r = z = 1$ and $k = 5$}
	\label{fig:rz1Mooretree}
\end{figure}
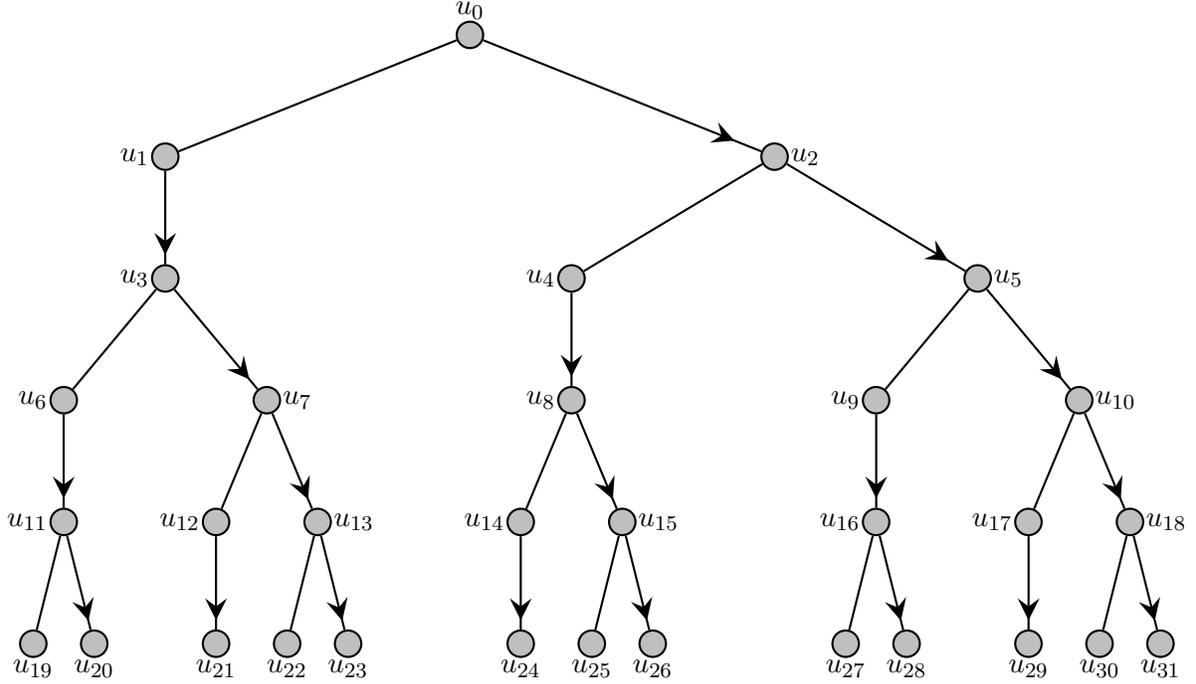
We begin with an elementary consideration on the length of the paths from a vertex to its repeat. 

\begin{lemma}\label{lem:shortpaths}
If there are two distinct mixed paths $P_1, P_2$ from $x$ to $y$ of length $\leq k$, then at least one of them has length $k$.
\end{lemma}
\begin{proof}
Suppose that there are mixed paths $P_1 \not = P_2$ from $x$ to $y$  and that both paths have length $\leq k-1$.  Draw the Moore tree of depth $k$ rooted at $x$.  $y$ appears twice in this tree, as does $y^+$, so that the defect would be at least two.
\end{proof}
Colloquially, there cannot be two `short' mixed paths between any pair of vertices.  

\begin{corollary}\label{cor:vnotrepeat}
No $v \in S$ is a repeat and $S$ is an independent set.
\end{corollary}
\begin{proof}
Let $v \in S$.  Suppose that there exists a vertex $u$ such that $v = r(u)$.  If $u \not = v$, then $u$ must have two paths of length $\leq k-1$ to $v^*$, contradicting Lemma~\ref{lem:shortpaths}.  If $u = v$, then $v$ is contained in a mixed cycle of length $\leq k$, so that there are again two short paths from $v$ to $v^*$.

By definition no arc can have both endpoints in $S$.  If there are $v, w \in S$ such that $v \sim w$, then no other vertex of $G$ would be able to reach $v$ or $w$.  Hence $S$ is independent.  
\end{proof}

We next deduce the relationship between repeats and the undirected neighbours of elements of $S$.

\begin{lemma}\label{lem:repeats}
For every vertex $u$ of $G$ and $v \in S$, either $v^* = r(u)$ or $v^* = r(u^*)$.
\end{lemma}
\begin{proof}
Draw the Moore tree rooted at $u$.  As $u_2 = u^+$ has a directed in-neighbour, obviously $u_2 \not \in S$.  $G$ has diameter $k$ and so, since $u_2$ can reach $v$ only through $v^*$, we have $d(u_2,v^*) \leq k-1$.  If $u = v$, then $u_1 = v^*$, so that $v^*$ occurs in both branches and $v^* = r(u)$.  If $v$ belongs to the $u_1$-branch, then either $v^*$ also lies in that branch or $u = v^*$, so that again $v^* = r(u)$.  We may thus assume that $v$ occurs only in the $u_2$-branch.

Suppose that $d(u_2,v) = k-1$.  Then $u_1$ cannot reach $v$ through $u$ and so $v^*$ occurs in both branches of the Moore tree.  Finally assume that $d(u_2,v) \leq k-2$, so that $d(u_2,v^*) \leq k-3$.  As $d(u_3,v) \leq k$, it follows that $d(u_3,v^*) \leq k-1$, so that $u_1=u^*$ has two mixed paths to $v^*$ with length $\leq k$, one through $u_3$ and the other through $u$ and $u_2$, so that $v^* = r(u^*)$.  
\end{proof}
This yields an upper bound on the size of $S$.

\begin{corollary}\label{cor:sizeofS}
$|S| \leq 2$.
\end{corollary}
\begin{proof}
Obviously $|S| \leq |\{ v^* : v \in S\} |$.  Fix a vertex $u$.  By Lemma \ref{lem:repeats}, $\{ v^* : v \in S\} \subseteq \{ r(u),r(u^*)\} $.
\end{proof}
As the average directed in-degree of vertices in $G$ is one, Corollary \ref{cor:sizeofS} shows that there are three possible situations:
\begin{enumerate}[label=\roman*)]
	\item $S = \{ v\}$, $S' = \{ v'\}$, $d^-(v') = 2$;
	\item $S = \{ v,w\}$, $S' = \{ v'\}$, $d^-(v') = 3$;
	\item $S = \{ v,w\}$, $S' = \{ v',w'\}$, $d^-(v') = d^-(w') = 2$. 
\end{enumerate}

We complete the proof of the theorem by examining these cases in turn.  Observe that in the final two cases, $v^* \not = w^*$, $\{ v^*,w^*\} \cap \{ v,w\} = \varnothing $ by Corollary \ref{cor:vnotrepeat} and $r(u) \in \{ v^*,w^*\} $ for all $u \in V(G)$ by Lemma \ref{lem:repeats}. 

\begin{theorem}
$(1,1,k;-1)$-graphs are totally regular for $k \geq 3$.
\end{theorem}
\begin{proof}
We refer the reader to Figure \ref{fig:rz1Mooretree} for clarity.  Suppose that option i) holds.  By Corollary \ref{cor:vnotrepeat} each vertex has a unique mixed path of length $\leq k$ to $v$.  We therefore obtain an upper bound on the order of $G$ by assuming that $v' \sim v$ and counting the vertices with paths of length $\leq k$ to $v$.  There are $M(1,1,k-2) = F_{k+1}-3$ vertices that can reach each of the directed in-neighbours of $v'$ by paths of length $\leq k-2$, so, counting also $v$ and $v'$, the maximum possible order of $G$ would be $2(F_{k+1}-3)+2 = 2F_{k+1}-4$, which is too small for $k \geq 3$. 

Now let us examine option ii).  We can see that $v'$ is a repeat as follows.  Draw the Moore tree rooted at $v'$.  We can assume that $r(v') \not = v'$, so that all directed in-neighbours of $v'$ lie at distance $k$ from $v'$.  One branch of the Moore tree must contain two elements of $Z^-(v')$, so that $v'$ is the repeat of a vertex in $N^+(v')$.

As all repeats lie in $\{ v^*,w^*\} $, we can set $v' = v^*$.  Hence $d^-(w^*) = 1$.  If $k = 3$, then $w$ has a unique in-neighbour $w^*$, which in turn has a unique directed in-neighbour; hence at most 6 vertices can reach $w$ by paths of length $\leq 3$, this bound being achieved when $v' \in N^-(w^*)$, whereas $M(1,1,3)-1 = 10$.  Thus $k \geq 4$.  Draw the Moore tree rooted at $v$.  Then $v_1 = v'$.  Neither $v_3$ nor $v_7$ is equal to $v'$, or there would be a cycle of length $\leq 2$ through $v'$, contradicting Lemma~\ref{lem:shortpaths}.  Hence, since $v_2$ can reach both of these vertices by paths of length $\leq k$, but all vertices in $N^-(v_3) = \{ v_1,v_6\} $ and $N^-(v_7) = \{ v_3,v_{12}\} $ already appear in the $v_1$-branch, $r(v)$ must occur in both $N^-(v_3)$ and $N^-(v_7)$, so that $v_1 = v'$ has two short paths to $r(v)$ in violation of Lemma~\ref{lem:shortpaths}.

Finally assume that iii) holds.  As in ii), we can show that both $v'$ and $w'$ are repeats, so we can take $v^* = v'$ and $w^* = w'$.  If $k = 3$, counting initial vertices of paths of length $\leq 3$ to $v$ shows that the order of $G$ is at most 9, whereas $|G| = 10$.  Similarly, if $k = 4$ the order is at most $16$, whereas $|G| = 18$.  Assume that $k \geq 5$ and let $u$ be an arbitrary vertex.  Consider $u_3, u_7, u_{11}$ and $u_{13}$.  A directed in-neighbour and the undirected neighbour of each of these vertices already occur in the $u_1$-branch.  However, $u_2$ can reach all of these vertices by paths of length $\leq k$.  Therefore
\begin{itemize}
	\item if $u_3 \not \in S'$, then $r(u) \in \{ u_1,u_6\} $;
	\item if $u_7 \not \in S'$, then $r(u) \in \{ u_3,u_{12}\} $;
	\item if $u_{11} \not \in S'$, then $r(u) \in \{ u_6,u_{19}\} $;
	\item if $u_{13} \not \in S'$, then $r(u) \in \{ u_7,u_{22}\} $.
\end{itemize}

By Lemma~\ref{lem:shortpaths} the above sets are disjoint, with the exception of $\{ u_1,u_6\} $ and $\{ u_6,u_{19}\} $.  Hence the five sets $\{ u_3,u_7\} , \{ u_3,u_{13}\} , \{ u_7,u_{11}\} , \{ u_7,u_{13}\} $ and $\{ u_{11},u_{13}\} $ intersect $S'$.  These vertices are distinct and $|S'| = 2$, so the only solution is $S' = \{ u_7,u_{13}\} $, $u_3,u_{11} \not \in S'$.  Thus $r(u) \in \{ u_1,u_6\} \cap \{ u_6,u_{19} \} $, so $r(u) = u_6$.  All repeats lie in $\{ v^*,w^*\} = S' = \{ u_7,u_{13}\} $, so $u_6 \in \{ u_7,u_{13}\} $, contradicting Lemma~\ref{lem:shortpaths}
\end{proof}

\section{Total regularity of $2$-geodetic mixed graphs with excess one}

We turn now to the question of the total regularity of mixed graphs with small excess.  A mixed graph is an $(r,z,k;+\epsilon )$-graph if it has minimum undirected degree $\geq r$, minimum directed out-degree $\geq z$, is $k$-geodetic and has order equal to $M(r,z,k) + \epsilon $; $\epsilon$ is the \emph{excess} of $G$.  We will show that $(r,z,2;+1)$-graphs are totally regular.  

Let $G$ be an $(r,z,2;+1)$-graph that is not totally regular.  $G$ has order $(r+z)^2+z+2$ and it can be shown in the manner of Lemma~\ref{lem:outregular} that $G$ is out-regular.  It follows that for every vertex $u$ of $G$ there is a unique vertex $o(u)$ that cannot be reached by a path of length $\leq 2$ from $u$.  We call $o(u)$ the \emph{outlier} of $u$.  The proof of the next lemma is similar to that of Lemma~\ref{lem:fundamental} and hence is omitted.  

\begin{lemma}\label{lem:fundamental3}
$\displaystyle S\subseteq\bigcap_{u \in V(G)} o(N^+(u))$, $\displaystyle S' \subseteq \bigcap _{u \in V(G)} N^+(o(u))$ and $d^-(v') = z+1$ for all $v' \in S'$.
\end{lemma}

\begin{lemma}\label{lem:outliersS'}
$S' = N^+(o(u))$ for all $u \in V(G)$ and $S \subseteq N^-(v')$ for all $v' \in S'$.
\end{lemma}
\begin{proof}
Fix $v' \in S'$.  By $2$-geodecity, every vertex has at most one path of length $\leq 2$ to $v'$.  By Lemma~\ref{lem:fundamental3}, $S'$ is contained in the out-neighbourhood of any outlier, so by $2$-geodecity no element of $S'$ is an outlier.  Hence we obtain a lower bound for the order of $G$ by assuming that $S \subseteq N^-(v')$ and counting paths of length $\leq 2$ to $v'$, yielding 
\[ |V(G)| \geq 1+r+z+1+r(r-1+z)+(z+1)(r+z)-|S'|.\]
Rearranging, we obtain $|S'| \geq r+z$.  By Lemma~\ref{lem:fundamental3}, $r+z$ is also an upper bound on the size of $S'$.  The resulting equality implies the second half of the result.
\end{proof}

\begin{theorem}\label{thm:totreg}
All $(r,z,2;+1)$-graphs are totally regular.
\end{theorem}
\begin{proof}
Trivially, $G$ contains at least two distinct outliers $o_1$ and $o_2$. If $G$ is not totally regular, then by Lemma~\ref{lem:outliersS'}, $N^+(o_1) = N^+(o_2)\,(=S')$.  Suppose that $r \geq 2$; then there are two distinct paths of length two from $o_1$ to $o_2$, contradicting $2$-geodecity.  Therefore $r = 1$.  However, this implies that $G$ contains a perfect matching and hence has even order, whereas the order $z^2+3z+3$ is odd.  Therefore $G$ must be totally regular.      
\end{proof}

\section{Total regularity of $(r,1,k;+1)$-graphs for $k \geq 3$}

As in the case of almost mixed Moore graphs, total regularity for $k \geq 3$ is a more difficult problem.  We solve it for $z = 1$.  Let $G$ be an $(r,1,k;+1)$-graph that is not totally regular, where $k \geq 3$.  Trivially $G$ is out-regular.  As before, for any vertex $u$ the unique directed out-neighbour of $u$ will be written $u^+$.  We first make an observation concerning the relationship between the position of vertices in the Moore tree rooted at $u$ and the outlier of the vertex $u^+$.

\begin{lemma}\label{lem:Mooretrees}
Draw the Moore tree of depth $k$ for an arbitrary vertex $u$.  Let $w$ be such that $d(u,w) \leq k-1$ and the mixed path from $u$ to $w$ does not begin with a directed arc (this includes the possibility $w = u$).  Then if either
\newline i) $w \in S$ or
\newline ii) $w \not \in S'$ and $w$ appears in the Moore tree as the endpoint of an arc,
\newline then $w = o(u^+)$.
\end{lemma}
\begin{proof}
Condition ii) is illustrated in Figure~\ref{fig:Mooretreelemma}, where dotted rectangles represent trees of depth $k - 1 \geq 2$. We suppose that $w\notin S'$; in-neighbours of $w$ are shown in black. If $w$ satisfies either of these conditions i) or ii), then, as in Figure~\ref{fig:Mooretreelemma}, every vertex of $N^-(w)$ appears in the undirected branches of the tree.  In particular, if $w = u \in S$, then $u$ has no in-neighbours apart from those in $U(u) = N^+(u)-\{ u^+\} $.  Therefore by $k$-geodecity there are no in-neighbours of $w$ within distance $k-1$ of $u^+$, so that $d(u^+,w) > k$.  
\end{proof}

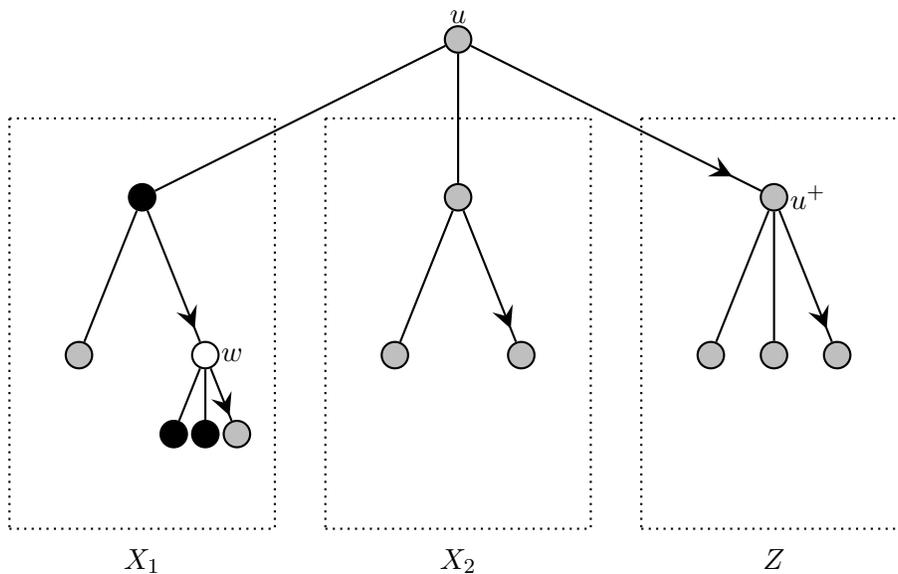
\begin{figure}\centering
	\begin{tikzpicture}[middlearrow=stealth,x=0.2mm,y=-0.2mm,inner sep=0.1mm,scale=2.1,
	thick,vertex/.style={circle,draw,minimum size=10,font=\small,fill=lightgray},every label/.style={font=\small}]
	\node at (200,0) [vertex,label=above:{$u$}] (v0) {};
	
	\node at (100,50) [vertex, fill=black] (v1) {};
	\node at (200,50) [vertex] (v2) {};
	\node at (300,50) [vertex,label=right:{$u^+$}] (v3) {};
	
	\node at (80,100) [vertex] (v9) {};
	\node at (120,100) [vertex,label=right:{$w$},fill=white] (v4) {};
	\node at (180,100) [vertex] (v5) {};
	\node at (220,100) [vertex] (v6) {};
	\node at (280,100) [vertex] (v10) {};
	\node at (300,100) [vertex] (v11) {};
	\node at (320,100) [vertex] (v12) {};
	
	\node at (110,125) [vertex, fill=black] (v13) {};
	\node at (120,125) [vertex, fill=black] (v14) {};
	\node at (130,125) [vertex] (v15) {};
	
	\node at (100,160) [label=below:{$X_1$}] {};
	\node at (200,160) [label=below:{$X_2$}] {};
	\node at (300,160) [label=below:{$Z$}] {};
	\path
	(v0) edge (v1)
	(v0) edge (v2)
	(v0) edge [middlearrow] (v3)
	(v1) edge (v9)
	(v1) edge [middlearrow] (v4)
	(v2) edge (v5)
	(v2) edge [middlearrow] (v6)
	(v3) edge (v10)
	(v3) edge (v11)
	(v3) edge [middlearrow] (v12)
	(v4) edge (v13)
	(v4) edge (v14)
	(v4) edge [middlearrow] (v15)
	;
	\draw[dotted] (58,25) -- (58,155);
	\draw[dotted] (58,25) -- (142,25);
	\draw[dotted] (58,155) -- (142,155);
	\draw[dotted] (142,25) -- (142,155);
	
	\draw[dotted] (158,25) -- (158,155);
	\draw[dotted] (158,25) -- (242,25);
	\draw[dotted] (158,155) -- (242,155);
	\draw[dotted] (242,25) -- (242,155);
	
	\draw[dotted] (258,25) -- (258,155);
	\draw[dotted] (258,25) -- (342,25);
	\draw[dotted] (258,155) -- (342,155);
	\draw[dotted] (342,25) -- (342,155);
	
	\end{tikzpicture}
	\caption{Diagram for Lemma~\ref{lem:Mooretrees}}
	\label{fig:Mooretreelemma}
\end{figure}

\begin{corollary}\label{outlierofv+}
$o(v^+) = v$ for all $v \in S$.
\end{corollary}
\begin{proof}
As $v \in S$, all of the in-neighbours of $v$ are contained in undirected branches of the Moore tree rooted at $v$.  Therefore by Lemma~\ref{lem:Mooretrees} the outlier of $v^+$ must be $v$ itself.
\end{proof}

We may now deduce the relative size of the sets $S$ and $S'$ and restrict the existence of edges and arcs in $S \cup S'$.

\begin{lemma}\label{lem:excessoneklarge}
$|S| = |S'|$ and $o(v') \in Z^-(v')$ for all $v' \in S'$.
\end{lemma}
\begin{proof}
Every vertex in $S$ has directed in-degree zero.  Let $v'$ be a vertex in $S'$ and consider the Moore tree of depth $k$ rooted at $v'$.  By $k$-geodecity, none of the $r+1$ branches of the tree can contain more than one in-neighbour of $v'$, so it follows that $d^-(v') = 2$ and $o(v')$ lies in $Z^-(v')$.  As the average directed in-degree of $G$ is one, $S$ and $S'$ must have the same size.
\end{proof}

\begin{lemma}\label{lem:Sindependent}
$S$ is an independent set.  For each $v \in S$ we have $v^+ \in S'$.  
\end{lemma}
\begin{proof}
By definition $S$ contains no arc.  Suppose that $v \sim w$, where $v, w \in S$.  Draw the Moore tree of depth $k$ rooted at $v$.  By Corollary~\ref{outlierofv+}, $o(v^+) = v$.  But as $w \in S$ is at distance $\leq k-1$ from $u$ and lies in an undirected branch, Lemma~\ref{lem:Mooretrees} implies that $w$ would also be an outlier of $v^+$, so that $v = w$, which is impossible.  Suppose now that $v \rightarrow w$, where $v \in S$ and $d^-(w) = 1$.  Let $u \in U(v)$, so that there is a path $u \sim v \to w$.  Applying Lemma~\ref{lem:Mooretrees} to $u$ shows that both $v$ and $w$ would be outliers of $u^+$, again a contradiction.  Thus $v^+ \in S'$ for all $v \in S$.    
\end{proof}

These lemmas allow us to readily deduce the total regularity of $(r,1,k;+1)$-graphs for $k \geq 4$.

\begin{theorem}\label{thm:totregkgeq4}
$(r,1,k;+1)$-graphs are totally regular for $k \geq 4$. 
\end{theorem}
\begin{proof}
Let $v \in S$.  We know from Lemma~\ref{lem:Sindependent} that $v$ has no out-neighbours in $S$ and that $v^+ \in S'$.  We now show that in fact all out-neighbours of $v$ lie in $S'$. Suppose for a contradiction that there is an edge $v \sim w$, where $d^-(w) = 1$.  Choose a path $x \sim y \to w$.  Applying Lemma~\ref{lem:Mooretrees} to the Moore tree rooted at $x$ shows that $v$ and $w$ would both be outliers of $x^+$, an impossibility.  Hence there can be no such edge and so $U(v) \subseteq S'$ for all $v \in S$. $S$ and $S'$ have the same size by Lemma~\ref{lem:excessoneklarge}, so we also have $U(v') \subseteq S$ for all $v' \in S'$.     

Let $v \in S$ and draw the Moore tree of depth $k$ rooted at $v$.  Recall that $o(v^+) = v$.  Let $v' \in S' \cap U(v)$.  $(v')^+$ appears as the endpoint of an arc in an undirected branch of the tree, so by Lemma~\ref{lem:Mooretrees} $(v')^+$ must be in $S'$, or it would be another outlier of $v^+$.  Hence $U((v')^+) \subseteq S$, so $(v')^+$ has an undirected neighbour $w$ in $S$ at level three of the tree.  This situation is depicted in Figure \ref{fig:Theorem5.5}.  By Lemma~\ref{lem:Mooretrees}, $w$ must be the outlier of $v^+$, so $w = v$, in violation of $k$-geodecity.  Therefore $G$ must be totally regular. 
\end{proof}

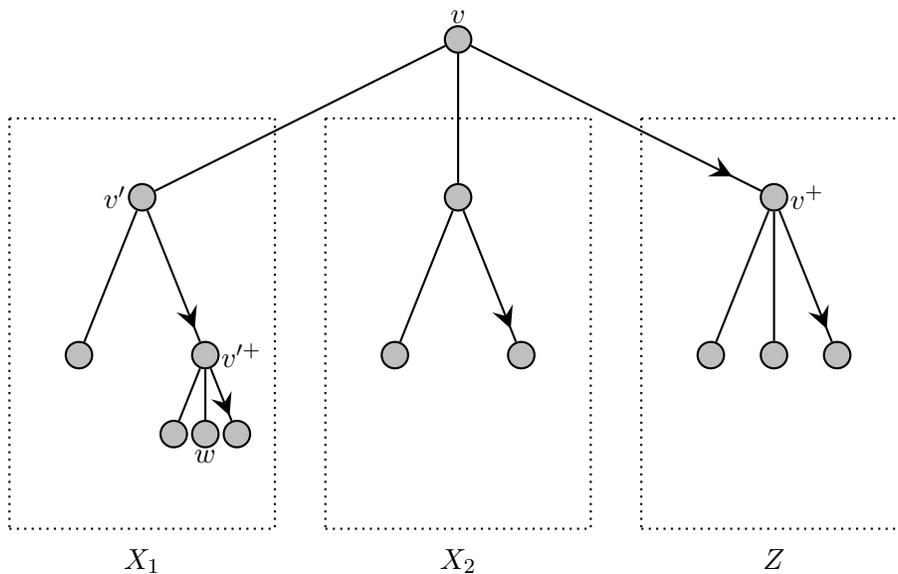
\begin{figure}\centering
	\begin{tikzpicture}[middlearrow=stealth,x=0.2mm,y=-0.2mm,inner sep=0.1mm,scale=2.1,
	thick,vertex/.style={circle,draw,minimum size=10,font=\small,fill=lightgray},every label/.style={font=\small}]
	\node at (200,0) [vertex,label=above:{$v$}] (v0) {};
	
	\node at (100,50) [vertex,label=left:{$v'$}] (v1) {};
	\node at (200,50) [vertex] (v2) {};
	\node at (300,50) [vertex,label=right:{$v^+$}] (v3) {};
	
	\node at (80,100) [vertex] (v9) {};
	\node at (120,100) [vertex,label=right:{$v'^+$}] (v4) {};
	\node at (180,100) [vertex] (v5) {};
	\node at (220,100) [vertex] (v6) {};
	\node at (280,100) [vertex] (v10) {};
	\node at (300,100) [vertex] (v11) {};
	\node at (320,100) [vertex] (v12) {};
	
	\node at (110,125) [vertex] (v13) {};
	\node at (120,125) [vertex,label=below:{$w$}] (v14) {};
	\node at (130,125) [vertex] (v15) {};
	
	\node at (100,160) [label=below:{$X_1$}] {};
	\node at (200,160) [label=below:{$X_2$}] {};
	\node at (300,160) [label=below:{$Z$}] {};
	\path
	(v0) edge (v1)
	(v0) edge (v2)
	(v0) edge [middlearrow] (v3)
	(v1) edge (v9)
	(v1) edge [middlearrow] (v4)
	(v2) edge (v5)
	(v2) edge [middlearrow] (v6)
	(v3) edge (v10)
	(v3) edge (v11)
	(v3) edge [middlearrow] (v12)
	(v4) edge (v13)
	(v4) edge (v14)
	(v4) edge [middlearrow] (v15)
	;
	\draw[dotted] (58,25) -- (58,155);
	\draw[dotted] (58,25) -- (142,25);
	\draw[dotted] (58,155) -- (142,155);
	\draw[dotted] (142,25) -- (142,155);
	
	\draw[dotted] (158,25) -- (158,155);
	\draw[dotted] (158,25) -- (242,25);
	\draw[dotted] (158,155) -- (242,155);
	\draw[dotted] (242,25) -- (242,155);
	
	\draw[dotted] (258,25) -- (258,155);
	\draw[dotted] (258,25) -- (342,25);
	\draw[dotted] (258,155) -- (342,155);
	\draw[dotted] (342,25) -- (342,155);
	
	\end{tikzpicture}
	\caption{Diagram for Theorem~\ref{thm:totregkgeq4} for $k \geq 4$}
	\label{fig:Theorem5.5}
\end{figure}

The case $k = 3$ is more challenging.  For the remainder of this section, assume $G$ to be an $(r,1,3;+1)$-graph that is not totally regular. We first establish an upper bound on the size of the sets $S$ and $S'$.

\begin{lemma}\label{lem:boundS}
	$|S| = |S'| \leq r+2$.
\end{lemma}
\begin{proof}
	Fix $v \in S$.  Drawing a Moore tree rooted at $v^+$ and applying Lemma~\ref{lem:Mooretrees}, we see that at most one element of $S$ is contained in $N^+(v^+)$, for any such vertex in $U(v^+)$ would be an outlier of $(v^+)^+$.  Now considering the Moore tree rooted at $v$, we see from Lemma~\ref{lem:Mooretrees} that any element of $S$ lying in an undirected branch of the tree lying at distance $\leq 2$ from $v$ would be an outlier of $v^+$, whereas we already know from Corollary~\ref{outlierofv+} that $o(v^+) = v$; it follows by $3$-geodecity that there can be at most $2$ vertices of $S$ at distance $\leq 2$ from $v$, including $v$ itself.

Now let $w$ be a vertex of $S$ that lies at distance $\geq 3$ from $v$.  There are more branches in the Moore tree rooted at $v$ than there are in-neighbours of $w$, so there must be an out-neighbour $x$ of $v$ such that the branch of the Moore tree associated with $x$ contains no in-neighbour of $w$, yielding $o(x) = w$.  Since $o(v^+) = v$, we can in fact say that $x \in U(v)$.  Therefore $\{ w \in S : d(v,w) \geq 3\} \subseteq \{ o(x) : x \in S\} $.  It follows that there are at most $r$ vertices of $S$ lying at distance $\geq 3$ from $v$.  In total, there are thus at most $r+2$ vertices in $S$.
\end{proof}

\begin{lemma}\label{thm:r2}
	$r = 2$. 
\end{lemma}
\begin{proof}
	Draw the Moore tree of depth $3$ rooted at some $v \in S $, numbering vertices in accordance with our convention.  Suppose that $r = 1$.  As $v_3$ is not the outlier of $v_2 = v^+$, we must have $v_3 \in S'$ by Lemma~\ref{lem:Mooretrees}.  By Lemma~\ref{lem:excessoneklarge}, $o(v_3) \not = v$, for otherwise there would be an arc from $v$ to $v_3$, so that $v_3$ can reach $v$ by a mixed path of length $\leq 3$.  As $v_1$ is the only in-neighbour of $v$, it follows that there would be a $\leq 3$-cycle through $v_1$.  Hence $r \geq 2$.

	There are exactly $r$ vertices that $v$ can reach by paths of length two consisting of an edge followed by an arc, each of which must lie in $S'$, as none are outliers of $v^+$.  Examine the  $r$ vertices that $v^+$ can reach by paths of this form; any of these vertices which do not lie in $S'$ must be an outlier of $(v^+)^+$, so at least $r-1$ of them belong to $S'$.  Along with $v^+$, we have thus identified at least $1+r+(r-1) = 2r$ elements of $S'$ in the tree.  By Lemma~\ref{lem:boundS}, it follows that $2r \leq r+2$.
\end{proof}
\begin{theorem}\label{thm:totreg3}
	$(r,1,3;+1)$-graphs are totally regular.
\end{theorem}
\begin{proof}
	By Lemma~\ref{thm:r2}, $r = 2$.  By the argument of the preceding theorem, $|S| =|S'| = 4$ and each element $v$ of  $S$ has four mixed paths of length $\leq 3$ to $S'$, namely an arc to $v^+$, a path of length three via $v^+$ and two paths of length two through the undirected neighbours of $v$.  As $o(v^+) = v$ for $v \in S$, distinct elements of  $S$ have distinct directed out-neighbours in $S'$, so there must be $v,w \in S$ and  $v' \in S'$ such that there are paths $v \sim x \to v'$ and $w \sim y \to v'$ for some vertices $x, y$.  It follows from Lemma~\ref{lem:Mooretrees} that $x \not = y$.  Every vertex of $S'$ has a directed in-neighbour in $S$ and $d^-(v') = 2$, which implies that there is an edge in $S$, contradicting Lemma~\ref{lem:Sindependent}.
\end{proof}

\section{The unique $(2,1,2;+1)$-graph}

In \cite{BusLopMir} it is proven that there is a unique almost mixed Moore graph with diameter $k = 2$ and degree parameters $r = 2, z = 1$.  At the time it was unknown whether such a graph must be totally regular.  Using our total regularity result Theorem~\ref{thm:totreg} we will now prove that there is a unique 2-geodetic mixed graph with the same degree parameters $r = 2, z = 1$ and excess $\epsilon = 1$.  

Let $D_{12} = \langle x,y : x^6 = y^2 = e, yxy^{-1}=x^{-1}\rangle$ be the dihedral group with order 12.  It is easily verified that the Cayley graph on $D_{12}$ with generating set $ \{ x^2, y,xy\} $, where the generator $x^2$ is associated with arcs and the involutions $y, xy$ with edges, is a $(2,1,2;+1)$-graph.  It is displayed in Figure \ref{fig:excessone}.  We will now prove that this $(2,1,2;+1)$-graph is unique up to isomorphism.  Hence let $G$ be an graph with these parameters.

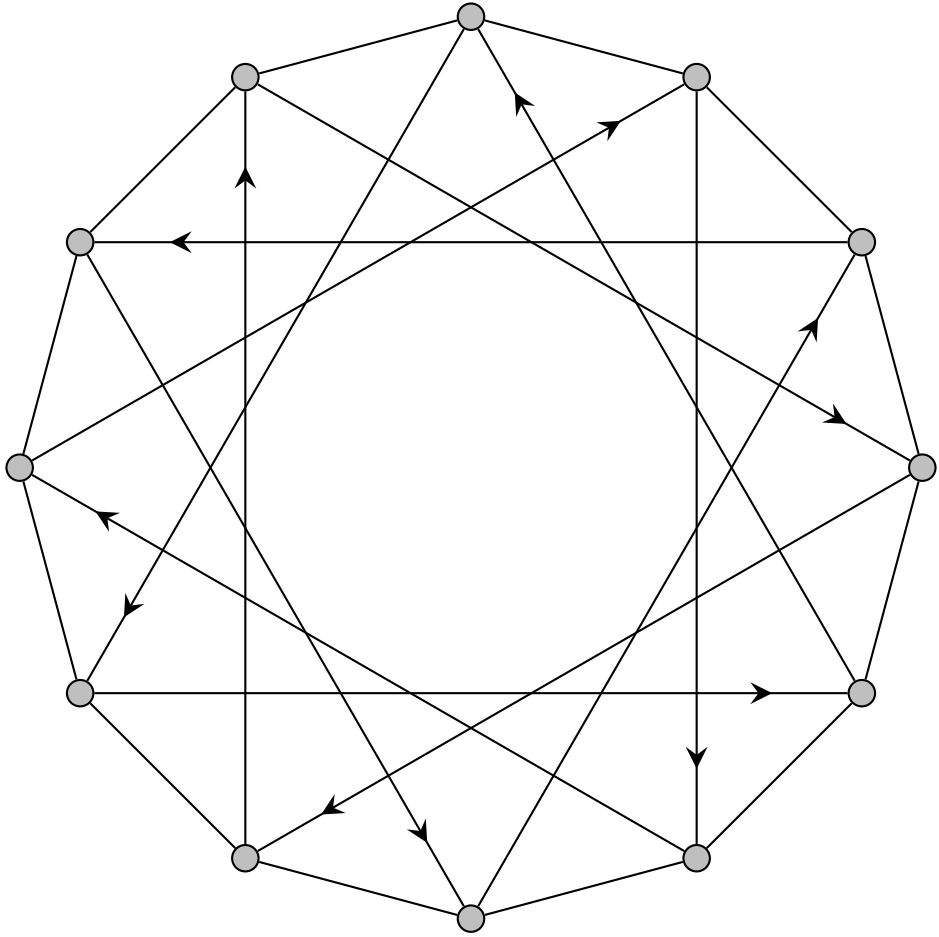
\begin{figure}[h]
	\centering
	\begin{tikzpicture}[middlearrow=stealth,x=0.2mm,y=-0.2mm,inner sep=0.2mm,scale=1,thick,vertex/.style={circle,draw,minimum size=10,fill=lightgray}]
	\node at (300,0) [vertex] (v0) {};
	\node at (259.8,150) [vertex] (v1) {};
	\node at (150,259.8) [vertex] (v2) {};
	\node at (0,300) [vertex] (v3) {};
	\node at (-150,259.8) [vertex] (v4) {};
	\node at (-259.8,150) [vertex] (v5) {};
	\node at (-300,0) [vertex] (v6) {};
	\node at (-259.8,-150) [vertex] (v7) {};
	\node at (-150,-259.8) [vertex] (v8) {};
	\node at (0,-300) [vertex] (v9) {};
	\node at (150,-259.8) [vertex] (v10) {};
	\node at (259.8,-150) [vertex] (v11) {};
	\path
	(v0) edge (v1)
	(v1) edge (v2)
	(v2) edge (v3)
	(v3) edge (v4)
	(v4) edge (v5)
	(v5) edge (v6)
	(v6) edge (v7)
	(v7) edge (v8)
	(v8) edge (v9)
	(v9) edge (v10)
	(v10) edge (v11)
	(v11) edge (v0)
	
	(v0) edge [middlearrow] (v4)
	(v4) edge [middlearrow] (v8)	
	(v8) edge [middlearrow] (v0)
	(v1) edge [middlearrow] (v9)
	(v9) edge [middlearrow] (v5)	
	(v5) edge [middlearrow] (v1)
	(v2) edge [middlearrow] (v6)
	(v6) edge [middlearrow] (v10)	
	(v10) edge [middlearrow] (v2)
	(v3) edge [middlearrow] (v11)
	(v11) edge [middlearrow] (v7)	
	(v7) edge [middlearrow] (v3)
	
	;
	\end{tikzpicture}
	\caption{The unique $2$-geodetic mixed graph with $r= 2, z = 1$ and excess $\epsilon = 1$}
	\label{fig:excessone}
\end{figure}

By Theorem \ref{thm:totreg}, $G^U$ is 2-regular and $G^Z$ is diregular with degree $z = 1$.  By 2-geodecity, $G^U$ cannot contain any cycle of length $\leq 4$.  Hence there are three possibilities: i) $G^U \cong C_5 \cup C_7$, ii) $G^U \cong 2C_6$ and iii) $G^U \cong C_{12}$. 

\begin{lemma}\label{lem:distancegeq4}
If $d_U(u,v) \leq 3$, then $u$ and $v$ are independent in $G^Z$.
\end{lemma}
\begin{proof}
The result is trivial if $d_U(u,v) = 1$.  If $d_U(u,v) = 2$, then any arc between $u$ and $v$ would violate 2-geodecity.  Similarly, if $d_U(u,v) = 3$ and $u \to v$, there is a path $u \sim u_1 \sim u_2 \sim v$, so that there are paths $u_1 \sim u \to v$ and $u_1 \sim u_2 \sim v$, which again is impossible.  
\end{proof}

\begin{lemma}\label{lem:GUcycle}
$G^U \cong C_{12}$.
\end{lemma}
\begin{proof}
Suppose that $G^U \cong C_5 \cup C_7$.  By Theorem \ref{thm:totreg}, there must be an arc with both endpoints in $C_7$.  However, these endpoints are at distance at most three in $G^U$, contradicting Lemma~\ref{lem:distancegeq4}.

Now let $G^U \cong 2C_6$.  We shall label the vertices of $G$ $(i,r)$, where $i \in \mathbb{Z}_6$ and $ r= 0,1$, so that $(i,r) \sim (i\pm 1,r)$, where addition is carried out modulo $6$.  By Lemma~\ref{lem:distancegeq4}, any arc must have its initial and terminal points in different cycles and by Theorem~\ref{thm:totreg} arcs that begin at distinct vertices of one cycle have distinct endpoints.  Without loss of generality, $(0,0) \to (3,1)$.  As $(0,0)$ can already reach the vertices $(1,0),(2,0),(5,0)$ and $(4,0)$ by paths in $G^U$ of length $\leq 2$, by 2-geodecity we must have $(3,1) \to (3,0)$.  Continuing in this manner, we deduce the existence of the 4-cycle $(0,0) \to (3,1) \to (3,0) \to (0,1) \to (0,0)$.  Now $(0,0)$ can reach the vertices $(3,1),(2,1)$ and $(4,1)$ by paths of length $\leq 2$, so either $(1,0) \to (1,1)$ or $(1,0) \to (5,1)$; by symmetry, we can set $(1,0) \to (5,1)$.  We now deduce the existence of the cycle $(1,0) \to (5,1) \to (4,0) \to (2,1) \to (1,0)$.  But now $(5,0) \not \to (4,1)$, or we would have $(0,0) \to (3,1) \sim (4,1)$ and $(0,0) \sim (5,0) \to (4,1)$, and likewise $(5,0) \not \to (1,1)$, or else $(4,0) \to (2,1) \sim (1,1)$ and $(4,0) \sim (5,0) \to (1,1)$.  It follows that $G^U$ is a twelve-cycle.
\end{proof}

\begin{theorem}
The $(2,1,2;+1)$-graph in Figure \ref{fig:excessone} is unique up to isomorphism.
\end{theorem}
\begin{proof}
By Lemma~\ref{lem:GUcycle}, $G^U$ is a 12-cycle.  We will label its vertices by the elements of $\mathbb{Z}_{12}$, so that $i \sim i\pm1$ for $i \in \mathbb{Z}_{12}$, where addition is modulo $12$.  By Lemma~\ref{lem:distancegeq4}, for each $i \in \mathbb{Z}_{12}$ we have $i \to i+r$, where $4 \leq r \leq 8$.  Suppose that for some $i$ we have $i \to i+6$, say $0 \to 6$.  Then $6$ is not adjacent to any vertex in $\{ 3,4,5,6,7,8,9\} $ by Lemma~\ref{lem:distancegeq4}.  Also $0$ can already reach every vertex in $\{ 10,11,0,1,2\} $ by undirected paths of length $\leq 2$, so by 2-geodecity the arc from $6$ also cannot terminate in this set.  Thus $i \not \to i+6$ for all $i \in \mathbb{Z}_{12}$.

Suppose now that there are vertices $u,v$ such that $d_U(u,v)=5$ and $u \to v$; without loss of generality, let $0 \to 5$.  By 2-geodecity, $5 \not   \to 10,11,0,1$ or $2$ and by Lemma~\ref{lem:distancegeq4} $5 \not \to 3,4,5,6,7$ or $8$.  Thus $5 \to 9$.  Consider the vertices that could be the directed in-neighbour of $0$.  By Lemma~\ref{lem:distancegeq4} none of the vertices $9,10,11,1,2$ or $3$ can have an arc to $0$.  The vertices $4,6$ and $7$ can already reach $5$ by undirected paths of length $\leq 2$, so, as $0 \to 5$, none of these vertices has an arc to $0$.  Therefore $8 \to 0$.  Finally we turn to the vertex $1$.  By  Lemma~\ref{lem:distancegeq4} $1$ cannot have an arc to any of $10,11,2,3$ or $4$.  If $1 \to 8$, then $1$ would have two paths $1 \sim 0$ and $1 \to 8 \to 0$ to $0$.  Similarly, if $1 \to 6$ then $1$ would have two paths to $5$.  Thus $1 \to 7$.  However, $7 = 1+6$, contradicting our previous result. 

Therefore for each $i \in \mathbb{Z}_{12}$ we have $i \to i\pm 4$.  By symmetry we can take $0 \to 4$, so that we have the 3-cycle $0 \to 4 \to 8 \to 0$.  We cannot have $1 \to 5$, or there would be two paths from $0$ to $5$ of length two.  Therefore $1 \to 9 \to 5 \to 1$.  Applying the same reasoning to the vertices $2$ and $3$, we deduce that $G^Z$ contains cycles $2 \to 6 \to 10 \to 2$ and $3 \to 11 \to 7 \to 3$.  By Theorem~\ref{thm:totreg} we have accounted for all edges and arcs, so it follows that $G$ is isomorphic to the graph in Figure \ref{fig:excessone}.   
\end{proof}

%----------------------------------------------
\section{Conclusion}

In this paper we have settled the long-standing open problem of the total regularity of mixed graphs with diameter two and defect one and $2$-geodetic mixed graphs with excess one.  The usefulness of these results is already demonstrated by the relative simplicity of the proof of the uniqueness of the $(2,1,2;1)$-graph in Figure \ref{fig:excessone}.  We hope that this paper will facilitate the search for graphs with defect or excess one for $k = 2$.  

However, the question of the total regularity of mixed graphs with order close to the Moore bound for larger values of $k$ has proven to be a more difficult problem  and remains largely open.  We conjecture:

\begin{conjecture}
All $(r,z,k;-1)$- and $(r,z,k;+1)$-graphs are totally regular.
\end{conjecture}

%----------------------------------------------

\end{document}